\documentclass{article}

\usepackage{url}
\usepackage[hidelinks]{hyperref}
\usepackage[utf8]{inputenc}
\usepackage[small]{caption}
\usepackage{graphicx}
\usepackage{amsmath}
\usepackage{booktabs}

\usepackage[utf8]{inputenc}
\usepackage{amsmath}
\usepackage{amsthm}
\usepackage{amsfonts}
\usepackage{amssymb}
\usepackage{natbib}
\setcitestyle{square}
\usepackage{tikz}
\usepackage{comment}
\usepackage{url}
\usepackage[ruled]{algorithm2e}
\makeatletter
\renewcommand{\@algocf@capt@plain}{above}
\makeatother
\usepackage{tabularx}
\usetikzlibrary{arrows}

\newtheorem{theorem}{Theorem}{}
{}
{}
\newtheorem{proposition}[theorem]{Proposition}{}
{}
\newtheorem{definition}[theorem]{Definition}{}

\newtheoremstyle{example}{\topsep}{\topsep}%
 {}
 {}
 {\bfseries}
 {}
 {\newline}
 {\thmname{#1}\thmnumber{ #2}\thmnote{ #3}}

\theoremstyle{example}
\newtheorem{example}{Example}{}

\title{Mixed-Integer Path-Stable Optimisation, with Applications in Model-Predictive Control of Water Systems}

\author{
Jorn Baayen\footnote{KISTERS Nederland B.V., St Jacobsstraat 123-135, 3511 BP Utrecht, The Netherlands. \newline E-mail: \url{jorn.baayen@kisters-bv.nl}.}\: and Jakub Marecek\footnote{IBM Research, Technology Campus Mulhuddart, Dublin 15, Ireland.\newline E-mail: \url{jakub.marecek@ibm.com}.}}

\date{\today}

\begin{document}

\maketitle

\begin{abstract}
Many systems exhibit a mixture of continuous and discrete dynamics.
We consider a family of mixed-integer non-convex non-linear optimization problems obtained in discretisations of optimal control of such systems. For this family, a branch-and-bound algorithm solves the discretised problem to global optimality. 

As an example, we consider water systems, where variations in flow and variations in water levels are continuous, while decisions related to fixed-speed pumps and whether gates that may be opened and closed are discrete.
We show that the related optimal-control problems come from the family we introduce -- and implement deterministic solvers with global convergence guarantees. 
\end{abstract}

\section{Introduction}

Many problems exhibit a mixture of continuous and discrete dynamics. Consider, for example, water systems: continuous variation of natural flow and continuous variations in channel water levels are in sharp contrast to discrete decision related to fixed-speed pumps and whether gates that may be opened and closed.
Optimization problems over such a mixture of continuous and discrete dynamics are an important challenge across optimal control, mathematical optimisation, and many engineering disciplines. 


A natural approach to optimise over continuous and discrete dynamics utilises finite discretisation of time and so-called mixed-integer non-linear optimisation solvers, 
which optimise over a finite set of discrete and continuous decision variables. 
Whereas many problems with continuous decision variables may be solved to global optimality efficiently, 
best known general-purpose algorithms for discrete problems or problems combining continuous and discrete decisions typically exhibit exponential worst-case complexity.

There are many approaches considered in the literature \citep[cf.]{floudas1995nonlinear,belotti2013mixed}. Heuristic approaches include linearizing the continuous dynamics and post-processing the results of a continuous optimization run, e.g., by rounding results to their nearest integer values.
Deterministic approaches with global convergence guarantees typically consider a sequence of progressively improving outer- or inner-approximations \citep[e.g.]{beale1976global,duran1986outer,floudas1995nonlinear,belotti2013mixed}.
The best-known such deterministic global optimization codes \citep{kronqvist2019review}
include Baron \citep{sahinidis1996baron},  Couenne \citep{belotti2009branching}, Antigone \citep{misener2014antigone}, and scip \citep{gleixner2017scip}.
Such general-purpose codes not only 
exhibit exponential worst-case complexity, but also
exhibit limited scalability in practice.

\vskip 3mm

In this paper, we present a class of mixed-integer non-linear optimization (MINLO) problems,
which we call path-stable, extending the notion of path stability introduced for continuous systems in \cite{Baayen2019}.
For such instances, a straightforward decomposition yields global optima of the MINLO. 
In the decomposition, the integer part of the optimization problem is solved by means of a branch-and-bound or branch-and-cut algorithm, whereas the continuous part is solved using a homotopy-continuation method \citep{bates2013numerically,allgower2012numerical}. 


Problems of this type are of great practical relevance when it comes to 
  controlling dams and weirs along the river, so as to reduce the amplitude of a flood wave by the time it reaches large urban centers. 
In some cases, one could apply non-linear model predictive control (MPC) of the non-linear hydraulic models \citep{ackermann2000real,Schwanenberg2015,Baayen2019}, but in in many other cases, 
the weirs consist of a small integer number of gates that are put in place or removed individually. 
One hence cannot control the flow at such a weir continuously.
One can, however, obtain a mixed-integer path-stable formulation.
Considering that extreme weather events appear  with an increasing frequency, causing more frequent and more intense flooding, improving the operations of water systems to reduce flood damage and save lives seems increasingly necessary. 

\section{Background and Related Work}



Before we present our approach, we should like to explain the related work we build upon. In Section \ref{sec:cont-nonlin} we provide a brief introduction to the classical homotopy-continuation method  for systems of non-linear equations.  In Section \ref{sec:parametric}, we show how the continuation method can be applied to find global optima of certain non-convex optimization problems.  
In Section \ref{sec:water}, we present related work in artificial intelligence and water engineering. 

\subsection{Solving Non-linear Equations}\label{sec:cont-nonlin}

Let us have a system of non-linear equations.
Let $F: \mathbb{R}^n \to \mathbb{R}^n$ denote its  residual function, such that we have:
\[
F(x)=0.
\]
Finding a solution $x^*$ such that $F(x^*)=0$ is undecidable over the reals. If an initial estimate $x_0$ is sufficiently close to $x^*$, the Newton-Raphson method converges to $x^*$. For other initial points, the Newton-Raphson method may diverge.

Instead of applying the Newton-Raphson method directly, one can approximate the residual function $F$ with a suitable function $\tilde{F}$, for which a solution $\tilde x^*$:
\[
\tilde{F}(\tilde x^*) = 0.
\]
is known. For any homotopy parameter $\theta \in [0, 1]$, we can use:
\[
G(x,\theta):=(1-\theta)\tilde{F} + \theta F
\]
as a deformation of $\tilde{F}$ into $F$.
Starting with $\tilde x^*$  such that $\tilde{F}(\tilde x^*)=G(\tilde x^*,0)=0$, we can increase $\theta$ progressively and solve $G(x,\theta)=0$ for $x$ starting from $\tilde x^*$. If the step in increasing $\theta$ is sufficiently small, Newton-Raphson method on  $G(\cdot,\theta)=0$ will converge. Eventually, we wish to arrive at a solution $x^*$ such that $F(x^*)=G(x^*,1)=0$.  

This process is known as tracing a path $\theta \mapsto x(\theta)$. By the implicit function theorem, this path exists locally and is unique as long as $\partial G / \partial x$ is non-singular.  Otherwise, the path may (1) turn back on itself, (2) end, or (3) bifurcate into multiple paths.
See \cite{bates2013numerically,allgower2012numerical} for further details.
In the following, we will look for conditions under which all points are non-singular.


\subsection{Parametric Programming}\label{sec:parametric}

One can generalise the path-tracing to parametric optimization problems of the form:
\begin{align}
 \min_x f(x,\theta) \quad & \nonumber\\
 \text{subject to} \quad & c(x,\theta) = 0, \tag{$\mathcal P^\theta$}\label{def:parametric-optimization-problem}\\
& x \geq 0 \nonumber;
\end{align}
where $x$ is the optimization variable, $\theta \in [0,1]$ is again a homotopy parameter, $f(x,\theta)$ is the objective function, and $c(x,\theta) : \mathbb R^{n} \times [0,1] \to \mathbb R^{\ell}$ denotes the set of constraints. 
Throughout, we assume that the functions $f$ and $c$ are at least twice differentiable.

Parametric optimization problems of this type have been studied extensively \citep[e.g.]{tikhonov1952systems,PT87,JW90,guddat1990parametric,massicot2019line}.  
One can consider an interior-point method \citep{forsgren2002interior} and solve a sequence of barrier problems:
\begin{gather}
\min_x f(x, \theta) - \mu \sum_{i=1}^n \ln (x_i)  \nonumber\\
\text{subject to} 
\quad c(x,\theta) = 0, \tag{$\mathcal P_\mu^\theta$}\label{def:parametric-barrier-optimization-problem}\\
\hspace*{15mm}
x \in \mathbb{R}^n; \nonumber
\end{gather}
for a non-negative decreasing sequence of barrier parameters $\mu$ converging to zero. Under suitable conditions, the optimum value of $(\mathcal P_\mu^\theta)$ converges to an optimal solution of $(\mathcal P^\theta)$ as $\mu$ tends to zero.

Consider the Lagrangian function $\mathcal L_\mu (x, \lambda, \theta)$ for the barrier problem (\ref{def:parametric-barrier-optimization-problem}), that is 
 $$\mathcal L_\mu (x, \lambda, \theta) := f(x, \theta) - \mu \sum_{i=1}^n \ln (x_i) + \lambda^T c(x,\theta),$$ where $\lambda$ is the vector of Lagrangian multipliers.
Then any solution of (\ref{def:parametric-barrier-optimization-problem}) is a solution of the  system of equations:
\begin{eqnarray}\label{eq:primal}
\nabla_{x} \mathcal L_\mu (x, \lambda, \theta)  & = & 0, \label{eq:primal1}\\
c(x, \theta) & = & 0. \nonumber \label{eq:primal2}
\end{eqnarray}

Let $F_\mu(x,\lambda,\theta)$ be the residual of the system of equations (\ref{eq:primal}). 
The system
\[
F_\mu(x,\lambda,\theta)=0
\]
admits a unique solution path in the neighborhood of $x^*$, $\lambda^*$, and $\theta^*$ if the Jacobian $\partial F_\mu(x^*,\lambda^*,\theta^*) / \partial (x,\lambda)$ is non-singular for the appropriate Lagrange multipliers $\lambda^*$.
On the other hand:

\begin{definition}[Singular point of  \cite{Baayen2019}]\label{def:singular-point}
For some fixed parameters $\hat \theta \in [0,1]$ and $\mu > 0$,
let $\hat x$ be a solution of the parametric barrier problem (\ref{def:parametric-barrier-optimization-problem}) and $\hat \lambda$ be the corresponding Lagrange multipliers.  The point $(\hat x,\hat \lambda,\hat \theta)$ is called a \emph{singular point} if the Jacobian $\partial F_\mu(\hat x,\hat \lambda , \hat \theta)/\partial (x, \lambda)$ is singular.
\end{definition}

To characterize the circumstances under which the parametric deformation process works well, i.e., does not admit singular points, \cite{Baayen2019} introduced:

\begin{definition}[Zero-convexity of  \cite{Baayen2019}]\label{def:zero-convex}
We say that the parametric optimization problem (\ref{def:parametric-optimization-problem}) is \emph{zero-convex} if the objective function $x \mapsto f(x,0)$ is a convex function, and the constraints $x \mapsto c(x,0)$ are linear.
\end{definition}

The notion of zero-convexity captures the idea that there should be a unique solution at $\theta=0$, and that it should be possible to find this solution using standard methods.  

\begin{definition}[Path stability of \cite{Baayen2019}] \label{def:path-stable}
We say that the parametric optimization problem (\ref{def:parametric-optimization-problem}) is \emph{path-stable with respect to the interior point method}
if its barrier formulation (\ref{def:parametric-barrier-optimization-problem}) does not admit singular feasible points for any $\mu > 0$ and any $\theta \in [0,1]$.
\end{definition}

In path stability, we seek a way of arriving at a uniquely related local minimum of the fully non-linear problem at $\theta=1$, i.e., 
without bifurcations along the way.
As it turns out, zero-convex and path stable problems can be solved to global optimality using a continuation algorithm of \cite{Baayen2019}, under the additional assumption of:

\begin{definition}[Path connectedness, e.g., \cite{mendelson1962introduction}]
A set $X$ is \emph{path-connected} if for any $x_1, x_2 \in X$, there exists a continuous function $f: [0,1] \to X$
such that $f(0)=x_1$ and $f(1)=x_2$.
\end{definition}

In the following, we extend these notions to a mix of continuous-valued and discrete-valued variables.

\subsection{Artificial Intelligence and Control of Water Systems}\label{sec:water}

In water systems, one typically considers one-dimensional movement of water in shallow channels, as modelled using the so-called Saint-Venant equations \citep{de1871theorie,vreugdenhil2013numerical}, which consist of the momentum equation
\begin{equation}\label{eq:saint-venant}
\frac{\partial Q}{\partial t} + \frac{\partial}{\partial x}\frac{Q^2}{A} + gA\frac{\partial H}{\partial x} + g\frac{Q|Q|}{A R C^2}=0,
\end{equation}
and continuity equation
\begin{equation}\label{eq:mass-balance}
\frac{\partial Q}{\partial x} + \frac{\partial A}{\partial t} = 0,
\end{equation}
where $x$ is the position along the channel,  $Q(x,t)$ is the discharge, $A(x,t)$ is the cross-sectional area, $H(x,t)$ is the water level, $R(x,t)$ is hydraulic radius (wetted perimeter), $C(x)$ is Chézy friction coefficient, and $g$ is the gravitational constant. Notice that the momentum introduces non-linearity, and that the continuity introduces non-linearity only if the cross sections $A$ are non-linear functions of the water level $H$, although this is almost always the case and certainly the case for natural water courses.

Within Artificial Intelligence (AI), there is a long history of monitoring such systems  \citep[e.g.]{dvorak1989model,schnitzler2015sensor}. Within control, much focus \citep[e.g.]{litrico2009modeling} is on feedback control. The optimal control is less studied, but no less important. Earlier approaches to optimal control subject to the Saint-Venant equations rely on local optimization \citep{ackermann2000real,Schwanenberg2015} or linearization \citep{amann2016online}. Global optimization over the Saint-Venant equations is considered by  \cite{Baayen2019}, where it is shown that appropriate discretizations are path-stable.
None of these approaches consider discrete decision variables, though.


\section{Our Approach}

We consider parametric mixed-integer non-linear optimization problems in the following standard form:
\begin{align}
\min_{x,\delta,\theta} f(x,\delta,\theta)  \label{eq:obj1} &  \\
\text{subject to } c(x,\delta,\theta) & = 0, \label{eq:ccons} \\
d(\delta) & = 0, \label{eq:dcons} \\
x & \geq 0,  \\
\delta_i & \in \{0,1\}. \label{eq:bin}
\end{align}
with continuous variables $x$, binary  variables $\delta$, which can be used to model bounded integer variables, and a continuation parameter $\theta$.

\begin{definition}
A \emph{relaxation} of the mixed-integer non-linear optimization problem  (\ref{eq:obj1}--\ref{eq:bin}) is
the continuous optimization problem obtained by omitting the integrality constraints \eqref{eq:bin}, and instead constraining the $\delta_i$ to either point values or subintervals within the closed interval $[0,1]$.
\label{def:relax}
\end{definition}

\begin{definition}
A mixed-integer non-linear optimization problem  (\ref{eq:obj1}--\ref{eq:bin}) is \emph{path-stable} with respect to the interior point method, when its relaxations are zero-convex and path-stable with respect to the interior point method, and their sets of feasible interior points are path-connected. 
\end{definition}

Notice that a necessary condition for path-stability is that the gradients of the constraints (\ref{eq:ccons}--\ref{eq:dcons}) are linearly independent (LICQ) across the feasible region (Proposition 3.6 of \cite{Baayen2019}).

\begin{figure}[bt]
\begin{tikzpicture}
\def\angle{90}%
\pgfmathsetlengthmacro{\xoff}{2cm*cos(\angle)}%
\pgfmathsetlengthmacro{\yoff}{0.55cm*sin(\angle)}%
\draw [thick, fill=gray!10,opacity=.2,text opacity=1] (\xoff,\yoff) circle[x radius=4.5cm, y radius=2cm] ++(2*\xoff,2*\yoff) node{{MI path-stable optimisation}};
\draw [thick, fill=gray!20,opacity=.2,text opacity=1] (0,-0.2) circle[x radius=1.6cm, y radius=1.6cm];
\draw [thick, fill=gray!30,opacity=.2,text opacity=1] (0,0.5) circle[x radius=0.7cm, y radius=0.7cm] node{MISC};
\draw [thick, fill=gray!30,opacity=.2,text opacity=1] (0,-0.95) circle[x radius=0.7cm, y radius=0.7cm] node{MILP};
\node [rectangle, rotate=90] (MIConvexPlace) at (-0.99, -0.1) {MI convex};
\node [rectangle] (ex1place) at (3, 0) {Example 2};
\node [rectangle] (ex1place) at (3, 1) {Example 1}; 
\end{tikzpicture}
\caption{Mixed-integer path-stable optimisation comprises much of mixed-integer (MI) convex optimisation, where certain constraint qualification (LICQ) holds, and certain non-convex problems, where LICQ holds.
See Examples \ref{ex1} and \ref{ex2} for instances that are not mixed-integer convex, but are mixed-integer path-stable.
Notice that mixed-integer linear programs (MILP) need not satisfy LICQ.
Notice that mixed-integer strictly-convex programs satisfying LICQ (MISC) are a proper subset of MI path-stable optimisation. 
}
\label{fig:generalises}
\end{figure}
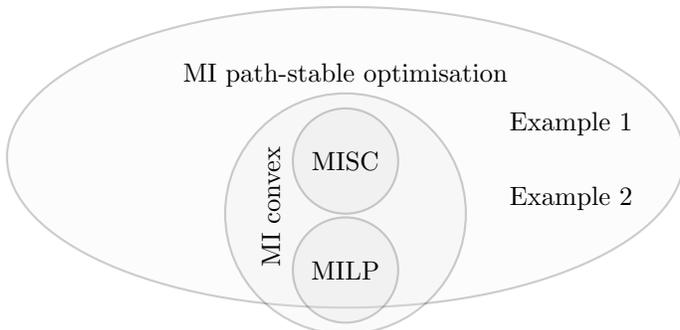

Notice also that a problem need not be convex, in order for it to be path-stable, as suggested in Figure~\ref{fig:generalises}. Let us consider two very simple examples:

\begin{example}[The unit circle]
\begin{align}
\max_{x_1,\delta_1} \; & x_1   \\
\text{subject to } &  (1-\theta)\left(x_1+\delta_1\right) + \theta \left( x_1^2 + \delta_1^2\right) - 1  = 0, \label{eq:circle} \\
& x_1  \geq 0.5, \label{eq:halfplane}  \\
& \delta_1  \in \{0,1\}. \label{eq:specificbin} 
\end{align}

We are interested in the solution of this problem at $\theta=1$.

The intersection of the unit circle defined by \eqref{eq:circle}, the half-plane defined by \eqref{eq:halfplane}, and the integrality constraint \eqref{eq:specificbin} is a single point $x_1^* = 1,\delta_1^* = 0$, which is also the optimum.
Notice that the continuous relaxation is non-convex for $\theta > 0$.

In this example, the bound \eqref{eq:halfplane} is crucial to ensure path-stability.  Without it, the gradient of the constraint \eqref{eq:circle} would vanish at $x_1=0,\delta_1=0$ for $\theta=1$.  The non-vanishing constraint gradient ensures LICQ, and its positive Lagrange multiplier (since we maximize $x_1$) ensures a non-singular Hessian of the Lagrangian of the barrier problem on the tangent space.  These two conditions characterize the path-stability of the relaxations. (Cf. Proposition 3.6 in \cite{Baayen2019}.)
\label{ex1}
\end{example}

\begin{example}[Branches of a parabola]\label{ex:2}
\begin{align}
\min_{x_1,x_2,x_3,\delta_1} & 0.001 x_2 + x_3^2   \\
\text{subject to } & (1-\theta)x_1 + \theta x_1^2 - x_2 - x_3 = 0, \label{eq:parabola} \\
& 2 - M(1-\delta_1)  \leq x_1 \leq -1 + M\delta_1,  \label{eq:bigM} \\
& -2 \leq x_1 \leq 3, \label{eq:xrange}  \\
& -1 \leq x_3 \leq 1, \label{eq:xrange2}  \\
& \delta_1 \in \{0,1\}. \label{eq:specificbin2} 
\end{align}

We are interested in the solution of this problem at $\theta=1$.

The Big-M constraint \eqref{eq:bigM} is readily transformed into two equality constraints using two slack variables $s_1,s_2 \geq 0$.  We do not carry out this transformation here in order not to obscure the geometric meaning of the example.

LICQ is trivially satisfied.  Since we are minimizing $x_2$, the Lagrange multiplier of the constraint \eqref{eq:parabola} is non-negative.  For any tangent space $T(x,\theta)$ of the constraint \eqref{eq:parabola}, any tangent vector $0 \neq y \in T(x,\theta)$ has $y_1 \neq 0$ or $y_3 \neq 0$.  Since the Hessian of the Lagrangian of the barrier problem is a diagonal matrix, and since $\frac{\partial^2 L_f}{\partial x_1^2} > 0$ and $\frac{\partial^2 L_f}{\partial x_3^2} > 0$, the Hessian is non-singular on any $T(x,\theta)$ of the constraint \eqref{eq:parabola}.

The Big-M constraint \eqref{eq:bigM} switches between inequality constraints $x_1 \leq -1$ if $\delta_1=0$, and $x_1 \geq 2$ if $\delta_1=1$.  The constraint \eqref{eq:parabola} relates the variables $x_1$ and $x_2$ in a parabolic fashion, with a ``penalty'' variable $x_3$.  

The penalty variable $x_3$ is introduced in order to ensure that replacing $x_3$ into the objective function results in a function that is not (quasi)convex.  To see this, note that its sublevel sets form sections of bands around the parabola (Figure \ref{fig:sublevel}). This example illustrates that there exist problems that are path-stable, but that do not become (quasi)convex once the constraints are eliminated.

\begin{figure}[!t]
\centering
 \includegraphics[scale=0.5]{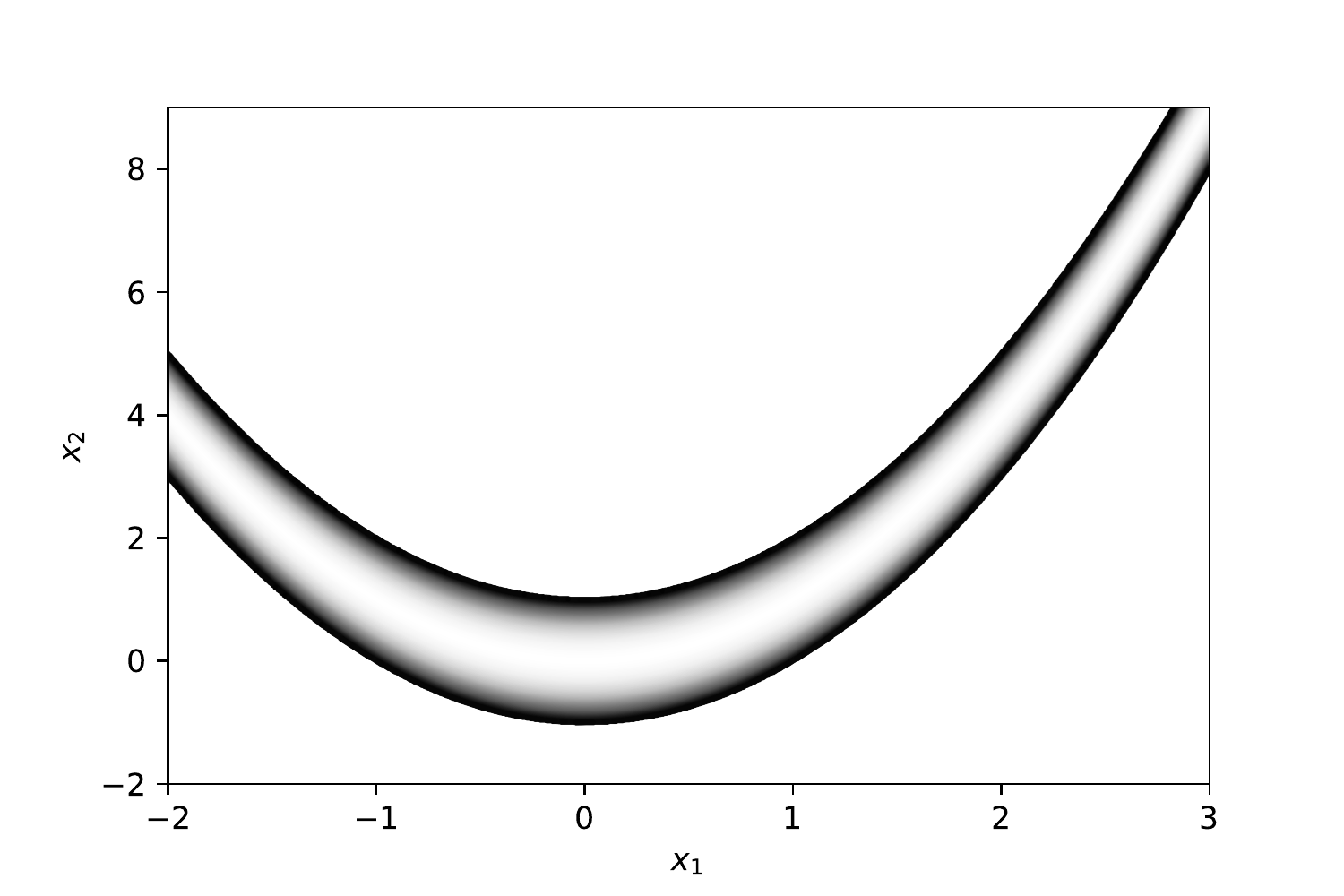}
\caption{Contour plot of the objective function of Example \ref{ex:2}, with the variable $x_3$ eliminated using the constraint \eqref{eq:parabola}.}
\label{fig:sublevel} 
\end{figure}

The objective essentially instructs the solver to stay on the parabola ($x_3=0$) and then to pick the smallest $x_2$ coordinate from two disconnected branches of a parabola.  The optimum lies on the left branch, since the constraint \eqref{eq:xrange} allows for smaller (in absolute value) $x_1$ and hence smaller $x_2$ on the left.
\label{ex2}
\end{example}

The above two examples are simple and would also have allowed us to prove directly that the Hessian of the Lagrangian of the barrier problem is positive definite on the tangent space of the constraint manifold at $\theta=1$.  For more complex problems, however, proving whether a Hessian is positive definite is often significantly harder than proving whether it is non-singular. 
See Proposition \ref{prop:river} for an example.

\subsection{A Decomposition}

We decompose the path-stable problem (\ref{eq:obj1}--\ref{eq:bin}) into two problems, a fully integral problem, and a fully continuous problem:
\begin{enumerate}
\item[INT] An integer-programming problem for the variables $\delta$ constrained by the constraint \eqref{eq:dcons} without continuous decision variables; and
\item[CNT] A continuous relaxation of Definition \ref{def:relax}, which is an optimization problem, where every integer decision variable is replaced either by a point value, or by a bounded continuous decision variable. That is: $\delta$ is constrained to lie in a hypercube instead of a discrete set.
\end{enumerate}

The proposed algorithm uses a branch-and-bound or branch-and-cut engine, such as \textsc{IBM ILOG CPLEX}, to perform implicit enumeration of the solutions of the integer-programming problem (INT).  At any point where the branch-and-cut engine needs an objective function value, it consults a continuation method (CNT), such as \textsc{KISTERS RTO}.  An overview is presented in Figure \ref{fig:my_label}.  The continuation solver is called when deciding whether to mark a node as the \emph{incumbent} (i.e., currently best) solution, and when making a branch pruning decision.

\emph{Incumbent nodes}.  The incumbent node is the node with the lowest objective value that has been visited by the algorithm thus far.  Since the objective value is obtained using CNT, a new node is marked as incumbent only if its CNT objective value is lower than the CNT objective value of the current incumbent.

\emph{Pruning decisions}.  A branch of nodes may be pruned if the CNT objective value corresponding to its partial solution is higher than the CNT objective value of the incumbent.  Due to branch pruning, not all nodes need to be visited. 

\emph{Solution of the initial system}. For a general non-convex problem, finding a feasible starting point of the homotopy continuation is hard.  
In mixed-integer path-stable optimisation, we 
start from a convex relaxation at $\theta=0$.
Continuation allows us to transport the positive Hessian eigenvalues of a (strictly) convex starting problem along with the parameter $\theta$ to the non-convex problem at $\theta=1$ \citep{Baayen2019}.

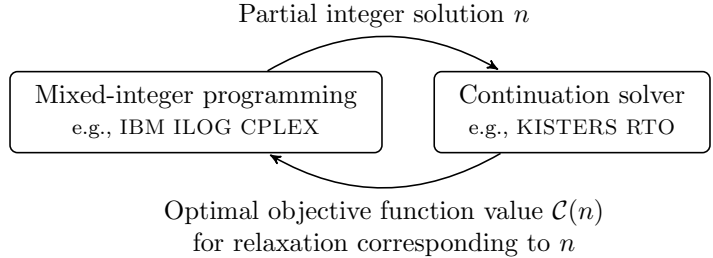
\begin{figure}[t]
    \centering
    
\begin{tikzpicture}[->,>=stealth',shorten >=1pt,auto,node distance=5.0cm,
                    semithick]

  \node[fill=white,draw=black,text=black,rounded corners=0.1cm] (miqp) {\begin{tabular}{c}Mixed-integer programming \\ \footnotesize e.g., \textsc{IBM ILOG CPLEX}\end{tabular}};
  \node[fill=white,draw=black,text=black,rounded corners=0.1cm] (homotopy)  [right of=miqp] {\begin{tabular}{c}Continuation solver \\ \footnotesize e.g., \textsc{KISTERS RTO}\end{tabular}};
  
  \path (miqp) edge [bend left] node {Partial integer solution $n$} (homotopy);
  \path (homotopy) edge [bend left] node {\begin{tabular}{c}Optimal objective function value $\mathcal{C}(n)$ \\ for relaxation corresponding to $n$\end{tabular}} (miqp);

\end{tikzpicture}
\caption{Information flows between the mixed-integer and continuation solvers.}
    \label{fig:my_label}
\end{figure}

The decomposition supports any type of constraint, as long as the problem (\ref{eq:obj1}--\ref{eq:bin}) is path-stable. 

\section{An Algorithm and its Convergence Analysis}

We will now present a bare-bones branch-and-bound algorithm \citep{floudas1995nonlinear}.  The branch-and-bound algorithm traverses the tree, wherein nodes represent partial assignments of values to integer decision variables. In other nodes, values of some of the integer decision variables are fixed, while other values are not.
In particular, leaves of the tree correspond to assignments to all integer decision variables.
While visiting a node $n$ of the tree, we invoke the homotopy-continuation solver on a relaxation (cf. Definition~\ref{def:relax}), where one subset of integer variables are assigned their fixed values, and the remainder of integer variables are treated as bounded continuous variables. $\mathcal{C}(n)$ denotes the objective-function value corresponding to the solution of the relaxation. 
See Algorithm \ref{alg:bnb} for the pseudo-code.


The convergence of Algorithm \ref{alg:bnb} is analyzed in Theorem \ref{thm}:

\begin{algorithm}[t]
\SetAlgoLined
\caption{Bare-bones branch-and-bound algorithm.}
\KwIn{A mixed-integer path-stable optimisation problem (\ref{eq:obj1}--\ref{eq:bin}).}
\KwResult{Global optimiser $\delta^*$ and the global optimum.}
 Initialise incumbent objective-function value to $B:=\infty$\;
 Initialise incumbent solution in terms of $\delta$ to an empty solution\;
 Initialize a queue to hold a partial integer solution $n$ with none of the integer variables of the problem assigned\;
 \While{queue is non-empty}{
  Take an arbitrary node $n$ off the queue\;
  \eIf{$n$ has all $\delta$ variables fixed}{ 
    \eIf{$\mathcal{C}(n) < B$}{
    Update the incumbent: Set $B:=\mathcal{C}(n)$ and $\delta:=n$ \;
     }{
     Discard $n$ as a solution candidate\;
     }
  }{
    \eIf{$\mathcal{C}(n) > B$}{
      Prune the branch\;
    }{
      Branch on an arbitrary element of 
     $\delta$ that is not fixed in $n$ and produce new nodes $n_i$\;
      Store the new nodes $n_i$ in the queue\;
    }
  }
 }
 \textbf{return} $\delta, B$\;
 \label{alg:bnb}
\end{algorithm}

\begin{theorem}\label{thm}
The branch-and-bound Algorithm \ref{alg:bnb} computes the global optimum of a path-stable mixed-integer non-linear optimization problem  (\ref{eq:obj1}--\ref{eq:bin}).
\end{theorem}
\begin{proof}
It suffices to show that no branches are pruned that may contain a global optimum. For a contradiction, assume that a branch $n$ is pruned that contains a global optimum $o$. 
Because the problem is path-stable, the relaxation is solved to global optimality by Theorem 3.8 of \cite{Baayen2019}. Hence, any global optimum contained in the branch must obtain an objective value that is no lower than that of the relaxation: $C(o) \geq C(n)$.  Per the pruning condition in Algorithm \ref{alg:bnb}, we have that
\[
C(o) \geq C(n) > B.
\]
This contradicts the presumed global optimality of $o$.  Hence, Algorithm \ref{alg:bnb} finds a global optimum of problem (\ref{eq:obj1}--\ref{eq:bin}).
\end{proof}

Notice that the algorithm runs in time exponential in the dimension of $\delta$ for the worst possible order of nodes taken from the queue, but that the average-case complexity may be lower, due to pruning.

\section{A Case Study: Cascaded River with Adjustable Weirs}

We consider a parametric river system with a varying number of cascaded reaches, some of which are terminated with an adjustable weir. 
The reaches are modelled using the Saint-Venant equations \citep{de1871theorie}, discretized using a semi-implicit scheme on a staggered grid as in \citep{casulli1998conservative,Baayen2019}. %
The grid is illustrated in Figure~\ref{fig:example-grid};
interpolation of control values between control time steps is linear.
An upstream inflow boundary condition is provided with a fixed time series.
Hydraulic parameters and further initial conditions are summarized in Table~\ref{table:example-specs}.
The model starts from steady state: the initial flow rate is uniform and the water level decreases linearly along the length of every reach.

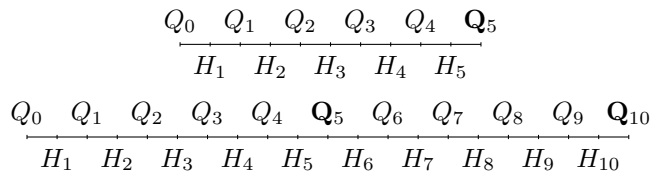
\begin{figure}[b]
\centering
\begin{tikzpicture}[scale=0.8]
\draw (0,0) -- (5.0,0);
\draw (0 cm, 1pt) -- (0 cm, -1pt) node[anchor=south] { \begin{tabular}{c} $Q_0$ \end{tabular} };
\foreach \x in {1,2,...,4}
    \draw (\x cm,1 pt) -- (\x cm,-1 pt) node[anchor=south] { \begin{tabular}{c} $Q_{\x}$ \end{tabular} };
\draw (5 cm, 1pt) -- (5 cm, -1pt) node[anchor=south] { \begin{tabular}{c} $\mathbf Q_5$ \end{tabular} };
\foreach \x in {1,2,...,5}
    \draw (\x cm - 0.5 cm,1 pt) -- (\x cm - 0.5 cm,-1 pt) node[anchor=north] {$H_{\x}$};
\end{tikzpicture}
\begin{tikzpicture}[scale=0.8]
\draw (0,0) -- (10.0,0);
\draw (0 cm, 1pt) -- (0 cm, -1pt) node[anchor=south] { \begin{tabular}{c} $Q_0$ \end{tabular} };
\foreach \x in {1,2,...,4}
    \draw (\x cm,1 pt) -- (\x cm,-1 pt) node[anchor=south] { \begin{tabular}{c} $Q_{\x}$ \end{tabular} };
\draw (5 cm, 1pt) -- (5 cm, -1pt) node[anchor=south] { \begin{tabular}{c} $\mathbf Q_5$ \end{tabular} };
\foreach \x in {6,7,...,9}
    \draw (\x cm,1 pt) -- (\x cm,-1 pt) node[anchor=south] { \begin{tabular}{c} $Q_{\x}$ \end{tabular} };
\draw (10 cm, 1pt) -- (10 cm, -1pt) node[anchor=south] { \begin{tabular}{c} $\mathbf Q_{10}$ \end{tabular} };
\foreach \x in {1,2,...,10}
    \draw (\x cm - 0.5 cm,1 pt) -- (\x cm - 0.5 cm,-1 pt) node[anchor=north] {$H_{\x}$};
\end{tikzpicture}
\caption{Staggered grid for the example problem with $1$ and $2$ adjustable weirs, respectively. Adjustable weir flows are located at $Q_5$ and $Q_{10}$ (indicated in bold).}
\label{fig:example-grid}
\end{figure}

\begin{table*}[t]
  \caption{Parameters for the example problem, loosely based on \cite{Baayen2019}.}
  \label{table:example-specs}
  \centering
  \begin{tabularx}{\linewidth}{ p{1cm} p{2.5cm} X}
  \toprule
     Param. & Value & Description \\
  \midrule
      $T$ & $24$ h & Optimization horizon \\
      $\Delta t^h$ & $10$ min & Hydraulic time step \\
      $\Delta t^c$ & $4$ h & Control time step \\
      $H^b_{i}$ & $( -4.90, -4.92$, $\ldots$, $-5.10)$ m & Bottom level, repeating for every reach \\
      $l$ & $10 000$ m & Length per reach \\
      $A_i(H_i)$ & $50\cdot(H-H^b_i)$ m$^2$ & Channel cross section function \\
      $P_i(H_i)$ & $50+2\cdot(H-H^b_i)$ m & Channel wetted perimeter function \\
      $C_{i}$ & $\left( 40, 40, \ldots, 40 \right)$ m$^{0.5}$/s & Chézy friction coefficient \\
      $\overline{H}$ & $0.0$ m & Nominal level in linear model for entire reach \\
      $\overline{Q}$ & $100$ m$^3$/s & Nominal discharge in linear model for entire reach \\
      $H_i(t_0)$ & $( 0.000, -0.025$, $\ldots$, $-0.222)$ m & Initial water levels at $H$ nodes, repeating for every reach \\
      $Q_i(t_0)$ & $\left( 100, 100, \ldots, 100 \right)$ m$^3$/s & Initial discharge at $Q$ nodes, repeating for every reach \\
      $Q_0$ & $100 \to 300$ m$^3$/s & Inflow hydrograph \\
      $\varepsilon$ & $10^{-12}$ & Absolute value approximation smoothness parameter \\
      $K$ & $10$ & Convective acceleration steepness factor \\
    \bottomrule
  \end{tabularx}
  
\end{table*}

Our optimization objective is to keep the water level at the $H$ nodes at $0$ m above datum:
\begin{align}
\min \sum_i \sum_{j} H_i(t_j)^2 \label{eq:obj}
\end{align}
subject to the binary adjustable weir flow constraint
\begin{align}
Q_k \in \{100, 200\} \label{eq:binary}
\end{align}
for all adjustable weir flow variables $Q_k$.  Adjustable weir flows, which are chosen at every control time step, are interpolated linearly onto the hydraulic time steps.

This problem is similar to the numerical example discussed in \cite{Baayen2019}, but differs in that we consider binary gates that can be either fully open or fully closed, leading to the binary adjustable weir flow constraint \eqref{eq:binary}. Still:

\begin{proposition}
For an optimal control problem of 
a parametric river system with a varying number of cascaded reaches modelled using Saint-Venant equations (\ref{eq:saint-venant}--\ref{eq:mass-balance}), with some reaches being terminated with a binary weir (\ref{eq:binary}), there exist a discretisation in the form of a mixed-integer path-stable optimisation problem.
\label{prop:river}
\end{proposition}

\begin{proof}
First, notice that the continuous-valued relaxation coincides with the continuous-valued problem of Proposition 4.6 of \cite{Baayen2019}, up to bounds on the individual variables obtained by branching on (\ref{eq:binary}).
Therefore we can apply the  staggered-grid semi-implicitly discretisation of (\ref{eq:saint-venant}--\ref{eq:mass-balance}), as suggested in equations (4.5--4.6) of \cite{Baayen2019}.
Next, consider the assumptions of Proposition 4.6 of \cite{Baayen2019}.
The assumption that none of the interior hydraulic variables are bounded (BND) 
is satisfied, because $Q_k$ once it has been branched on, becomes a fixed boundary for the up- and downstream reaches.
That is: fixing value of a scalar $Q_k$ to a numerical constant eliminates the weir flow  from the continuous-valued problem, which is split into two at that binary weir $k$.
The remaining assumptions (ICO, HBC, QBC, and OBJ) are satisfied trivially.
The result hence follows from Proposition 4.6 of \cite{Baayen2019}. 
\end{proof}


\begin{figure}[!tp]
\centering
 \includegraphics[scale=0.8]{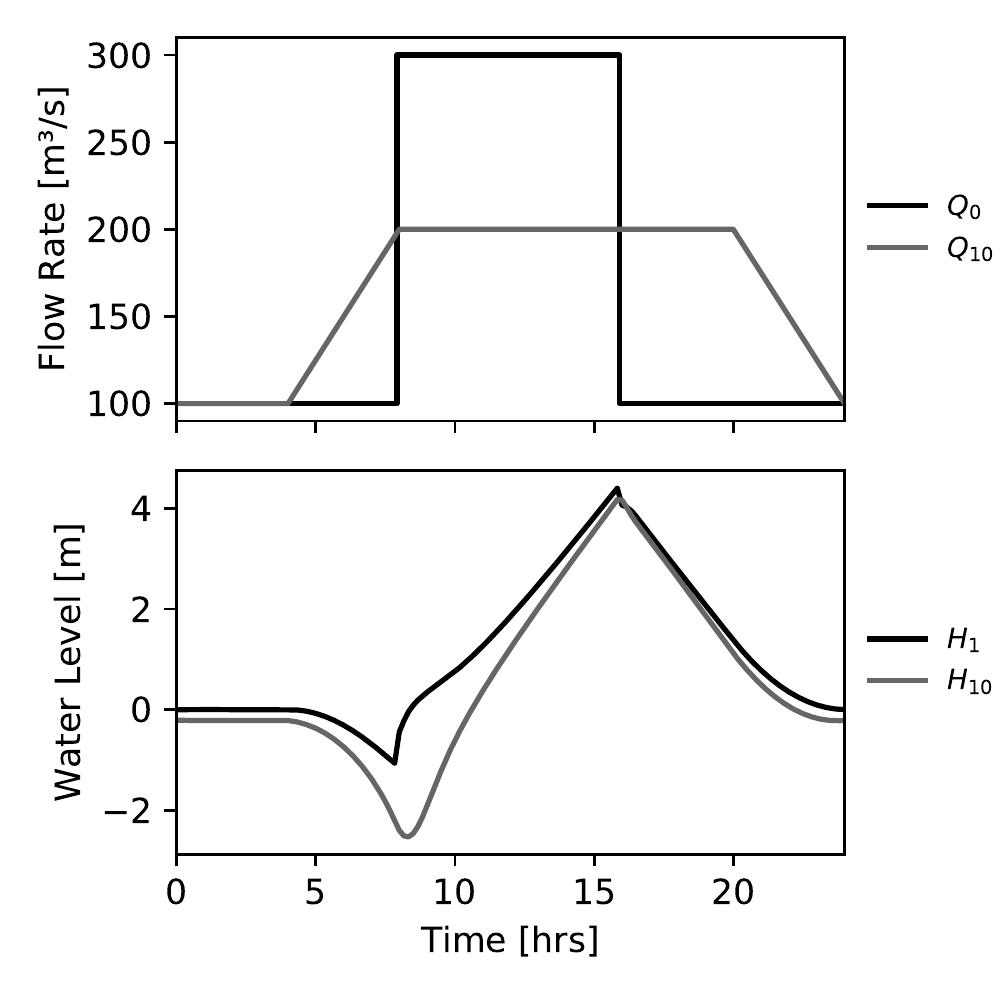}
\caption{Solution of the example optimization problem for a single weir.}
\label{fig:results_1} 
\end{figure}

\begin{figure}[!tp]
\centering
 \includegraphics[scale=0.8]{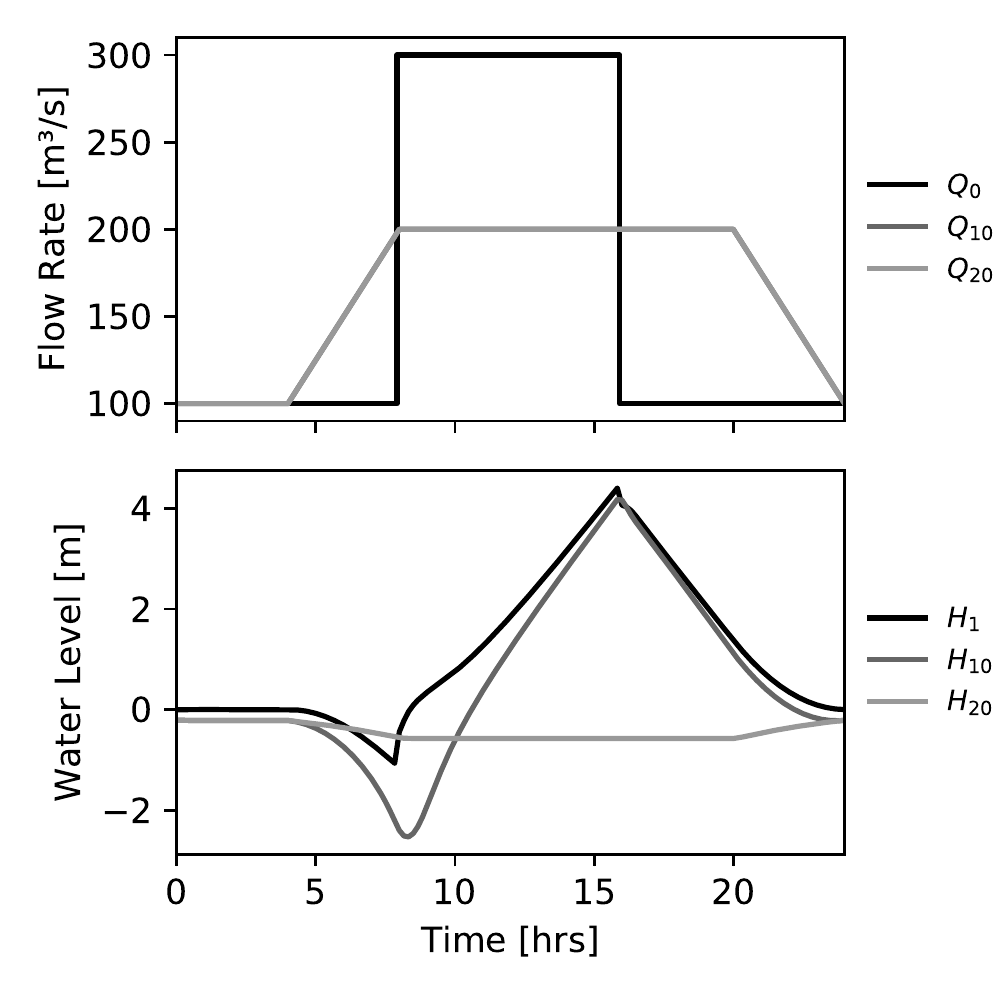}
\caption{Solution of the example optimization problem for two weirs. At the 4 hourly control time resolution, the optimal solution is to pass the outflow of the upstream reach right through the downstream reach.}
\label{fig:results_2} 
\end{figure}

The solution to the optimization problem is plotted in Figures~\ref{fig:results_1} and \ref{fig:results_2} for one and two adjustable weirs, respectively.
By releasing water in anticipation of the inflow using the decision variables, the optimization is able to reduce water-level fluctuations and keep the water levels close to the target level.

This optimization problem was implemented in Python using the \textsc{CasADi} package \citep{andersson2019casadi} for algorithmic differentiation
and \textsc{IBM ILOG CPLEX} \citep{cplexv1291} as the the branch-and-bound framework.  
On a MacBook Pro with 2.9 GHz Intel Core i5 CPU, the one and two weir cases take approximately $42$ s and $136$ s to solve using 4 CPU cores, respectively.
The complete source code is available on-line\footnote{\url{https://github.com/jbaayen/homotopy-example/tree/cplex-coupling}}.

\section{Concluding Remarks}

In this paper, we have demonstrated the viability of global mixed-integer optimization subject to non-linear dynamics under certain novel sufficient conditions. 
A decomposition was presented that, subject to the set of sufficient conditions, ensures that a branch-and-bound or branch-and-cut-type algorithm converges to a global optimum of the mixed-integer problem.

The approach was illustrated with examples of a flow in a cascaded channel with adjustable weirs.  
A variety of extensions thereof in scheduling hydro-power and flood control are of considerable commercial and societal value.

For perhaps the most prominent example in flood control, consider the automated drainage pump control by the Rijnland water authority in the Netherlands.
The homotopy-continuation method of \cite{Baayen2019} is already in use 
\citep{Bosbo1,Baayen2019-2} for the polders of Amsterdam Schiphol airport and the surrounding area. At the moment, only the continuous dynamics are considered, 
but it would be natural to consider a pump flow set point as a semi-continuous variable, which in the definition of \cite{cplexv1291,ghaddar2015lagrangian} 
can either be 0 or any value between its semi-continuous lower bound and its upper bound.  Such semi-continuous variables may give rise to mixed-integer path-stable optimisation problems.

\section*{Acknowledgments}
Jakub would like to acknowledge most pleasant 
discussion with Bradley Eck,
Vyacheslav Kungurtsev, and Sean McKenna.

\bibliographystyle{named}

\newpage

\bibliography{main}

\end{document}